\documentclass[12pt]{amsart}

\usepackage{amsmath}
\usepackage{amssymb}
\usepackage{amsfonts}
\usepackage{amsthm}
\usepackage{enumerate}
\usepackage{hyperref}
\usepackage{color}
\usepackage[all]{xy}

\theoremstyle{plain}
\newtheorem{thm}{Theorem}[section]

\newtheorem{cor}[thm]{Corollary}

\newtheorem{conj}[thm]{Conjecture}

\theoremstyle{definition}
\newtheorem{dfn}[thm]{Definition}

\newtheorem{dfns-rems}[thm]{Definitions and Remarks}
\newtheorem{notas-rems}[thm]{Notations and Remarks}
\newtheorem{exmps-rems}[thm]{Examples and Remarks}

\begin{document}


\title[$M$-Shellability of Discrete Polymatroids]
{$M$-Shellability of Discrete Polymatroids}


\author[M. Alizadeh]{Majid Alizadeh}

\address{Majid Alizadeh, School of Mathematics, Statistics and Computer Science,
college of Science, University of Tehran, P. O. Box 14155-6455,
Tehran,Iran.}

\email{malizadeh@khayam.ut.ac.ir}

\author[A. Goodarzi]{Afshin Goodarzi}

\address{Afshin Goodarzi, School of Mathematics, Statistics and Computer Science,
College of Science, University of Tehran, Tehran, Iran.}

\email{af.goodarzi@gmail.com}

\author[S. Yassemi]{Siamak Yassemi}

\address{Siamak Yassemi, School of Mathematics, Statistics and Computer Science,
College of Science, University of Tehran, Tehran, Iran, and School
of Mathematics, Institute for Research in Fundamental Sciences
(IPM), P.O. Box 19395-5746, Tehran, Iran.}

\email{yassemi@ipm.ir}

\urladdr{http://math.ipm.ac.ir/yassemi/}

\begin{abstract} In this note we show that every discrete polymatroid is
$M$-shellable. This gives, in a partial case, a positive answer to
a conjecture of Chari and improves a recent result of Schweig
where he proved that the $h$-vector of a lattice path matroid
satisfies a conjecture of Stanley.

\end{abstract}


\subjclass[2000]{05B35, 05E45}


\keywords{Discrete polymatroid, $h$-vector, Shellability}


\thanks{}


\maketitle


\section{Introduction and Preliminaries} \label{sec1}

A \emph{matroid} $M$ is a pair $(E(M),\mathcal{B}(M))$ consisting
of a finite set $E(M)$ and a collection $\mathcal{B}(M)$ of
subsets of $E(M)$, called \emph{bases} of $M$, that satisfy the
following two conditions:
\begin{itemize}
\item[(B1)] $\mathcal{B}(M)\neq \emptyset$, and
\item[(B2)] for each pair of distinct sets $B$, $B'$ in
$\mathcal{B}(M)$ and for each element $x\in B\setminus B'$, there
is an element $y\in B'\setminus B$ such that $(B-x)\cup y$ is in
$\mathcal{B}(M)$.
\end{itemize}

Subsets of bases are called \emph{independent sets}. The
collection of independent sets of a matroid form an abstract
simplicial complex, called \emph{matroid complex}.

 For a $(d-1)$-dimensional simplicial complex $\Delta$, let $f_i$ be the
number of $(i-1)$-dimensional faces of $\Delta$ (i.e. the faces of
cardinal $i$), and $f(\Delta)=(f_0, f_1, \ldots, f_{d})$ its
$f$-vector. The $h$-vector
 $h(\Delta)=(h_0,h_1, \ldots, h_{d})$ is defined by
 $H(y)=F(y-1)$, where $H(y)=\sum_{i=0}^{d} h_{i}y^{d-i}$ and
$F(y)=\sum_{i=0}^{d} f_{i} y^{d-i}$.

A \emph{monomial order ideal} $\Gamma$ on a set $V=\{x_1, \ldots,
x_n\}$ of variables is a set of monomials $x_{1}^{a_1} \ldots
x_{n}^{a_n}$ such that $u\in \Gamma$ and $v|u$ imply that $v\in
\Gamma$. The \emph{degree sequence} of $\Gamma$ is
$h(\Gamma)=(h_{0}, h_{1}, \ldots)$, where $h_{i} = \#\{u \in\Gamma
| \mathrm{deg} u=i\}$. We will not distinguish between a monomial
order ideal and its poset (ordered by divisibility).

A \emph{pure $M$-vector} is the degree sequence of an order ideal
of monomials, whose maximal elements have the same degree.

The following conjecture of Stanley~\cite{kn:ST} is one of the
most important conjectures on $h$-vector of matroid complexes.

\begin{conj}\label{conj1} \textrm{$($Stanley$)$} The $h$-vector of a matroid complex is a
pure $M$-vector.
\end{conj}

A poset $Q$ is an $M$-\emph{poset} if there exists a monomial
$M$ on a finite set $E$ of indeterminates (variables) such that $Q$ is
isomorphic to the poset  (ordered by divisibility) on the set of monomials on $E$ that divide
$M$. Equivalently, an $M$-poset is a direct product of chains.

Given two elements $x\leq y$ of a poset $P$, the interval $[x,y]$
is called an $M$-\emph{interval} if it is an $M$-poset. A pure
poset $P$ is called $M$-\emph{partitionable} if $P$ can be
partitioned into $M$-intervals $[x_1,y_1], \ldots, [x_n,y_n]$ such
that for each $1\leq i\leq n$, $y_i$ is a maximal element of the
poset $P$. Such a partition is called an
$M$-\emph{partition} of the poset $P$.

\begin{dfn}\label{dfn1} An $M$-\emph{shelling} of a poset $P$ is an $M$-partition of
$P$ along with an ordering of the $M$-intervals such that the
union of the elements in any initial subsequence of $M$-intervals
is an order ideal of $P$. A poset $P$ is $M$-\emph{shellable}, if
it admits an $M$-shelling.
\end{dfn}

Chari \cite{kn:ch} proposed a stronger version of Stanley's
conjecture for $h$-vectors of matroid complexes based on the
concept of $M$-shellability:

\begin{conj}\label{conj2} \textrm{$($Chari$)$} The $h$-vector of a matroid complex is a
shellable $M$-vector.
\end{conj}

Recall that a pure $M$-vector is called shellable if it is the
degree sequence of an $M$-shellable order ideal of monomials.

Herzog and Hibi~\cite{kn:hh} introduced discrete polymatroid,
which it is a generalization of matroids. Let $\Gamma$ be a pure
monomial order ideal on the variables $\{x_1, \ldots, x_r\}$ and
for any $m\in\Gamma$, the degree of $x_{i}$ in $m$ is denoted by
$m_i$. We say $\Gamma$ is a discrete polymatroid if, for any two
maximal monomials $m,m'\in\Gamma$ and index $i$ with
$m_{i}>m'_{i}$, there exists an index $j$ such that $m_j<m'_j$ and
$\frac{x_j}{x_i}m\in\Gamma$, cf.~\cite[Definition 4.1.]{kn:js}.

The aim of this paper is to show that every discrete polymatroid
is $M$-shellable (Theorem~\ref{thm1}). We apply this
 result to show that the $h$-vector of a lattice path matroid
(see Section~\ref{sec2} for definition) satisfies
Conjecture~\ref{conj2}.


\section{Main Theorem}\label{sec2}

\begin{thm}\label{thm1} Every discrete polymatroid is $M$-shellable.
\end{thm}
\begin{proof} Let $\Gamma$ be a discrete polymatroid on the
set $\{x_1, \ldots, x_r\}$ of variables and let $p$ be the number
of maximal elements of $\Gamma$. The proof is by
induction on $p$. If $p=1$,the basic case, then $\Gamma$ is an $M$-poset and the
assertion is obvious. So assume that $p>1$. Then there exist
 an index $j$ and two maximal elements $m$ and $m'$ in
$\Gamma$ with $m_j\neq m'_j$. With no lose of generality, we
assume that $j=r$. Now, put
\begin{itemize}
\item $k=\max\{m_r \ | \ m\in\Gamma\}$;
\item  $\Gamma_{1}=\{m\in\Gamma \ | \ x_{r}^k
\nmid m\}$;
\item $\Gamma_{2}=\Gamma -\Gamma_{1}$; and
\item $\Gamma'=\{\frac{m}{x_{r}^{k}} \ | \  m\in\Gamma_2\}$.
\end{itemize}

\noindent {\bf Claim:} $\Gamma_1$ and $\Gamma'$ are discrete
polymatroids.

\noindent{\bf Proof of Claim:} We only show that $\Gamma_1$ is discrete polynomial. A similar argument works for $\Gamma'$. First note that $\Gamma_1$ is a monomial order ideal. Since,  for $m\in\Gamma_{1}$ and $u\mid m$ we get that
$x_{r}^k\nmid u$ which implies that $u\in\Gamma_1$.To prove the purity of $\Gamma_1$, we assume that  this is not the case and get a contradiction. By assumption, there exist a maximal element $m$ in $\Gamma_1$ and an element
 $m'\in\Gamma_2$ such that $m\mid m'$. So $m'=x_{r}^t m$, for
some $t>0$. Let $m^{''}$ be a maximal elements in $\Gamma$ with
$m''_r<k$. Then there exists an index $j$ such that
$\frac{x_j}{x_i}m'=x_j x_{r}^{t-1}m\in\Gamma$ which it is contradict $m$ is a maximal element of $\Gamma_1$. Thus
$\Gamma_1$ is pure. To complete the proof we assume that $m$ and $m'$ be two monomials
in $\Gamma_1$ with $m_i>m'_i$, for some $i$. Then there exists an index $j$ such that
$m_j<m'_j$ and $\frac{x_j}{x_i}m\in\Gamma$, since $\Gamma$ is a
discrete polymatroid. If $j\neq r$, then
$x_{r}^k \nmid \frac{x_j}{x_i}m$, since $x_{r}^k \nmid m$. For $j=r$ we have $m_r<m'_r<k$ and  then $(\frac{x_j}{x_i}m)_r<k$.
Therefore $\Gamma_{1}$ is a discrete polymatroid. This complete the proof of the claim.

By induction hypothesis, there  exist  the
following $M$-shelling orders for $\Gamma_1$ and $\Gamma'$:
\begin{center}
$\Gamma_{1}=[a_1,b_1]\dot{\cup}\cdots\dot{\cup}[a_n,b_n]$\  \ and \ \
$\Gamma'=[c_1,d_1]\dot{\cup}\cdots\dot{\cup}[c_l,d_l]$.
\end{center}

We claim that the following order
\begin{center}
$\Gamma=[a_1,b_1]\dot{\cup}\cdots\dot{\cup}[a_n,b_n]\dot{\cup}
[x_{r}^{k}c_1,x_{r}^{k}d_1]\dot{\cup}\cdots\dot{\cup}[x_{r}^{k}c_l,x_{r}^{k}d_l]$
\end{center}
is an $M$-shelling for $\Gamma$. It suffices to show that every
initial subsequence
$A=\Gamma_1\dot{\cup}[x_{r}^{k}c_1,x_{r}^{k}d_1]\dot{\cup}\cdots\dot{\cup}[x_{r}^{k}c_s,x_{r}^{k}d_s]$
($s<l$) is an order ideal. Assume the contrary. Then there exist
$m\in A-\Gamma_1$ and $u\in\Gamma- A$ with $u\mid m$. Therefore, $\frac{u}{x_{r}^k}\in[c_1,d_1]\dot{\cup}\cdots\dot{\cup}[c_s,d_s]$. Since
$\frac{u}{x_{r}^k}\mid \frac{m}{x_{r}^k}$, and $\Gamma'$ is $M$-shellable.
It contradicts $u\in \Gamma- A$. Now the proof is complete.
\end{proof}

Note that the converse of Theorem~\ref{thm1} does not hold. As a counterexample, one can
 consider the monomial order ideal $\Sigma$ with maximal
elements $xy$ and $z^2$. It is easy to see that $\Sigma$ is
$M$-shellable but it is not a discrete polymatroid.

A sequence $(h_0, h_1,\ldots, h_r)$ is called a $PM$-\emph{vector}
if it is the degree sequence of some discrete polymatroid. Clearly,
every $PM$-vector is a pure $M$-vector. But Theorem~\ref{thm1}
gives the following generalization of this fact.

\begin{cor}\label{cor2} Every $PM$-vector is a shellable $M$-vector.
\end{cor}

The $h$-vector of $\Sigma$ in the example before Corollary ~\ref{cor2} is $(1, 3, 2)$. It
shows that $(1, 3, 2)$ is a shellable $M$-vector, but it is indeed
a $PM$-vector (take the discrete polymatroid with maximal
elements $xy$ and $yz$). However we guess these two classes of
vectors are very closed.

We end the paper by a result on lattice path matroids.

\noindent Fix two lattice paths $P=p_{1}p_{2} \ldots p_{m+r}$ and
$Q=q_{1}q_{2} \ldots q_{m+r}$ from $(0,0)$ to $(m,r)$ with $P$
never going above $Q$. For every lattice path $R$ between $P$
and $Q$, let $\mathcal{N}(R)$ be the set of $R$'s north steps.

In \cite{kn:bmn}, the authors showed that
$M[P,Q]=\{\mathcal{N}(R): R$ is a path between $Q$ and $P\}$ is a
matroid. $M[P,Q]$ is called a \emph{lattice path matroid}.

Schweig \cite[Theorem 3.6.]{kn:js} showed that lattice path
matroids satisfy Conjecture~\ref{conj1}. Even more, he proved that
the $h$-vector of a lattice path matroid is a $PM$-vector,
~\cite[Corollary 4.5.]{kn:js}. This result of Schweig and
Corollary~\ref{cor2} together imply the following result, which
says that lattice path matroids satisfy Conjecture~\ref{conj2}.
\begin{cor}\label{cor3} The $h$-vector of a lattice path matroid
is a shellable $M$-vector.
\end{cor}






\end{document}